\documentclass{birkjour}

\usepackage[utf8]{inputenc}
\usepackage{amsmath, amsthm, amsfonts, amssymb}
\usepackage{hyperref} 
\usepackage{enumitem}

\newtheorem{theorem}{Theorem}

\newtheorem{coro}[theorem]{Corollary}
\newtheorem{prop}[theorem]{Proposition}
\theoremstyle{definition}
\newtheorem{remark}[theorem]{Remark}

\newcommand{\eps}{\varepsilon}
\renewcommand{\phi}{\varphi}

\newcommand{\N}{\mathbb{N}}

\newcommand{\ipr}[1]{\left\langle#1\right\rangle}
\newcommand{\set}[1]{\Bigl\{#1\Bigr\}}
\newcommand{\bset}[1]{\Biggl\{#1\Biggr\}}
\newcommand{\be}{\begin{equation}}
\newcommand{\ee}{\end{equation}}
\newcommand{\co}{\colon}

\newcommand{\calL}{\mathcal{L}}

\DeclareMathOperator{\Span}{span}
\DeclareMathOperator{\rank}{rank}
\DeclareMathOperator{\codim}{codim}

\begin{document}

%
%
%
%
%
%
%
%
%

\title[Inequalities between s-numbers]{Inequalities between s-numbers}

\author{
Mario Ullrich}
\address{Institut f\"ur Analysis, 
Johannes Kepler Universit\"at Linz, Austria
} 
\email{mario.ullrich@jku.at}



\date{September 16, 2014}


\dedicatory{In memory of Albrecht Pietsch}


\begin{abstract} 
Singular numbers of linear operators between Hilbert spaces were generalized to Banach spaces by s-numbers (in the sense of Pietsch). 
This allows for different choices, including 
approximation, Gelfand, Kolmogorov and Bernstein numbers. 
Here, we present an elementary proof of a bound between the smallest and the largest s-number. 
\end{abstract}

\maketitle

\medskip

We start with introducing the terminology and a presentation of the results. 
In Section~\ref{sec:hist} we will discuss them and some
history.  
Proofs are given in Section~\ref{sec:proof}.

\section{s-numbers}

In what follows, let $X$, $Y$, $Z$ and $W$ be real or complex Banach spaces.  
The (closed) unit ball of $X$ is denoted by $B_X$, 
the dual space of $X$ by $X'$, and the identity map on $X$ is denoted by $I_X$. 
For a closed subspace $M\subset X$, we write $J_M^X$ for the embedding 
$J_M^X\co M\to X$ with $J_M^X(x)=x$, 
and $Q_M^X$ for the canonical map $Q_M^X\co X\to X/M$ with $Q_M^X(x)=x+M$
onto the quotient space 
$X/M:=\{x+M\co x\in X\}$ with norm $\|x+M\|_{X/M}:=\inf_{m\in M}\|x+m\|_X$.
The dimension of a subspace $M\subset X$ is denoted by $\dim(M)$, and by $\codim(M):=\dim(X/M)$ we denote its codimension.

The class of all bounded linear operators between Banach spaces 
is denoted by $\calL$, and by $\calL(X,Y)$ we denote those operators from $X$ to $Y$, equipped with the operator norm. 
The rank of an operator $S\in\calL(X,Y)$ is defined by $\rank(S):=\dim(S(X))$.

\medskip
\goodbreak

A map $S\to (s_n(S))_{n\in\N}$ assigning to every operator $S\in\calL$ a non-negative scalar sequence $(s_n(S))_{n\in\N}$ is called an 
\textit{s-number sequence} if, for all $n\in\N$, the following conditions are satisﬁed

\begin{enumerate}[label=(S\arabic*)]
\itemsep=2mm
\item $\|S\|=s_1(S)\ge s_2(S) \ge \dots \ge 0$ \qquad for all \; $S\in\calL$,
\item $s_n(S+T) \,\le\, s_n(S)+\|T\|$ \qquad for all \; $S,T\in\calL(X,Y)$,
\item $s_n(BSA) \,\le\, \|B\| \, s_n(S) \, \|A\|$ \qquad where \;$W \stackrel{A}{\longrightarrow} X\stackrel{S}{\longrightarrow} Y\stackrel{B}{\longrightarrow} Z$, 
\item $s_n(I_{\ell_2^n})=1$, 
\item $s_n(S)=0$ \quad whenever $\rank(S)<n$.
\end{enumerate}
\medskip
We call $s_n(S)$ the \textit{$n^{\rm th}$ s-number} of the operator $S$. To indicate the underlying Banach spaces, we sometimes write $s_n(S\co X\to Y)$.

\goodbreak

There are some especially important examples of s-numbers: 
\begin{itemize}
    \item approximation numbers: 
    \[
    a_n(S) \,:=\, \inf\set{\|S-L\|\co L\in\calL(X,Y), \; \rank(L)<n}
    \]
    \item Bernstein numbers: 
    \[
    b_n(S) \,:=\, \sup\set{\inf_{x\in M\setminus\{0\}}\frac{\|Sx\|}{\|x\|}\co M\subset X, \; \dim(M)=n}
    \]
    \item Gelfand numbers: 
    \[
    c_n(S) \,:=\, \inf\set{\|S J_M^X\|\co M\subset X, \; \codim(M)<n} 
    \] 
    \item Kolmogorov numbers: 
    \[
    d_n(S) \,:=\, \inf\set{\|Q_N^Y S\|\co N\subset Y, \; \dim(N)<n}
    \]
    \item Weyl numbers: \\[-4mm]
    \[
    x_n(S) \,:=\, \sup\bset{\frac{a_n(SA)}{\|A\|}\co A\in\calL(\ell_2,X),\; A\neq 0}
    \]
    \item Hilbert numbers: 
    \[
    \begin{split}
    h_n(S) \,:&=\, \sup\bset{\frac{a_n(BSA)}{\|B\|\|A\|}\co A\in\calL(\ell_2,X), \; B\in\calL(Y,\ell_2),\; A,B\neq0}. \\
    \end{split}
    \]
\end{itemize}

\bigskip 

We refer to~\cite{Pie87,Pie07} for a detailed treatment of the above, and a few other, s-numbers and their specific properties.  

\begin{remark} \label{rem:def}
The original definition of s-numbers in~\cite{Pietsch-s} 
used the stronger \textit{norming property} 
$(S4')\co s_n(I_X)=1$ for all $X$ with $\dim(X)\ge n$. 
This did not allow for $x_n$ and $h_n$, 
and has been weakened in~\cite{Bauhardt,Pietsch-Weyl} for defining 
them, leading to the least restrictive axioms that still imply uniqueness for Hilbert spaces, see Proposition~\ref{prop1}.  
It is sometimes assumed that the s-number sequence is \textit{additive}, i.e., 
$(S2')\co s_{m+n-1}(S+T)\le s_m(S)+s_n(T)$, 
or \textit{multiplicative}, i.e., 
$(S3')\co s_{m+n-1}(ST)\le s_m(S) s_n(T)$. 
Both properties hold for $a_n,c_n,d_n,x_n$, 
while $h_n$ is only additive, 
and $b_n$ is neither of the two, 
see~\cite{Pie07,Pie08}. 
\end{remark}

\goodbreak 

The following proposition is well-known, see~\cite[2.3.4 \& 2.6.3 \& 2.11.9]{Pie87}. 

\begin{prop} \label{prop1}
For every s-number sequence $(s_n)$, $S\in\calL$ and $n\in\N$, we have 
\[
h_n(S) \,\le\, s_n(S) \,\le\, a_n(S). 
\]
Equalities hold if $S\in\calL(H,K)$ for Hilbert spaces $H$ and $K$.
\end{prop}


For convenience, 
we present a sketch of the proof of the inequalities in Section~\ref{sec:proof}. 
Using only elementary arguments, 
we prove the following reverse inequality,  
which is known in a 
more involved form 
based on an intermediate comparison with $x_n(S)$, see~\cite[2.10.7]{Pie87} or~\cite[6.2.3.14]{Pie07}, 
or Remark~6 in~\cite{KNU24}.


\begin{theorem}\label{thm:c-h}
For all $S\in\calL$ and $n\in\N$, we have 
\[
\max\set{c_n(S),\, d_n(S)} 
\;\le\; n\,\left(\prod_{k=1}^n h_k(S) \right)^{1/n}.
\]    
\end{theorem}

\medskip

Since $c_n(I\co \ell_1\to\ell_\infty)\asymp 1$ 
and $h_n(I\co \ell_1\to\ell_\infty)\asymp n^{-1}$, 
see~\cite[6.2.3.14]{Pie07} and~\cite{HL84}, 
this result is best possible up to constants.

\bigskip 
\goodbreak

We cannot obtain bounds for individual $n$ from Theorem~\ref{thm:c-h}, 
see also Remark~\ref{rem:eig}, 
but combining it with the inequality 
$n^\alpha\le e^\alpha\, (n!)^{\alpha/n}$ 
for $\alpha\ge0$, 
we obtain a 
more handy form. 
Moreover, the known fact that 
$a_n(S)\le (1+\sqrt{n})\, c_n(S)$, 
see~
\cite[2.10.2]{Pie87}, leads to a bound between $a_n$ and $h_n$, and hence between all s-numbers.  

\medskip

\begin{coro}
For all $S\in\calL$, $\alpha>0$ and $n\in\N$, we have 
\[
c_n(S) \,\le\, e^\alpha\, n^{-\alpha+1}\cdot 
\sup_{k\le n}\, k^\alpha\, h_k(S),
\]
and
\[
a_n(S) \,\le\, 2\,e^\alpha\, n^{-\alpha+3/2}\cdot 
\sup_{k\le n}\, k^\alpha\, h_k(S).
\]
\end{coro}





\section{A bit of history} \label{sec:hist}

We provide a brief description 
of the relevant facts from 
(the highly recommended) 
``History of Banach Spaces
and Linear Operators'' of Pietsch~\cite{Pie07}, 
and give some further references.

\textbf{Singular numbers} 
of operators 
on Hilbert spaces 
have become 
fundamental 
tools in (applied) mathematics 
since their introduction in 1907 by Schmidt~\cite{Schmidt1907}. 
For compact $S\in\calL(H,K)$ between complex Hilbert spaces~$H,K$, 
the \textit{singular numbers} are 
defined by 
$s_k(S):=\sqrt{\lambda_k(SS^*)}$, 
where the \textit{eigenvalues} $\lambda_k(T)$ of $T\in\calL(X,X)$ are characterized by $Te_k=\lambda_k(T)\cdot e_k$ for some $e_k\in X\setminus\{0\}$, and ordered decreasingly. 
%
Applications range from 
the study of \textit{eigenvalue distributions} of operators,  
see~\cite{Koe86} or~\cite[6.4]{Pie07}, 
to the classification of \textit{operator ideals}~\cite[6.3]{Pie07}, 
to the singular value decomposition, aka Schmidt representation, 
with its many applications.

\textbf{\!\!\!s-numbers} are a generalization to 
linear operators between Banach spaces. 
However, 
there is no unique substitute for singular numbers but, 
depending on the context, different 
s-numbers may be used to 
\textit{quantify compactness}, 
while 
others may be easier to compute. 
As Pietsch wrote in~\cite[6.2.2.1]{Pie07}, 
``we have a large variety of s-numbers that make our life more interesting.''

Most notably, 
$a_n$, $b_n$, $c_n$ and $d_n$ were already known in the 1960s, 
sometimes in a related form as ``width'' of a set, 
see~\cite{Ism74,Tikh60} or~\cite[6.2.6]{Pie07}, 
and are 
by now part of the foundation of approximation theory~\cite{Pinkus85} 
and information-based complexity~\cite{Mathe90,Novak}. 
More recent treatments of the subject 
and extensions 
can be found, e.g., in~\cite{FMS21,LE11,SX24}. 
Let us also highlight~\cite{KNU24}, where 
we discuss the relation of s-numbers to \textit{minimal approximation errors} in detail, 
and use variants of Theorem~\ref{thm:c-h} 
to bound the \emph{maximal gain} of randomized/adaptive algorithms over 
deterministic/non-adaptive ones. 

An axiomatic theory of s-numbers has been developed by Pietsch~\cite{Pietsch-s,Pie87,Pie07} in the 1970s. 
This, in particular, allowed for a characterization of the smallest/largest s-number (with certain properties), 
see also Remark~\ref{rem:def}. 

\textbf{Inequalities} 
and several relations between s-numbers have already been collected 
in~\cite{Pietsch-s}, see also~\cite[2.10]{Pie87} and~\cite[6.2.3.14]{Pie07}. 
A particularly interesting bound is 
$d_n(S)\le n^2 b_n(S)$, which was proven 
by Mityagin and Henkin~\cite{MH63} in 1963. 
They also conjectured that $n^2$ can be replaced by $n$, see also~\cite[p.~24]{Pinkus85}. 
This bound with $b_n$ replaced by $h_n$, 
and the corresponding conjecture, have been given in~\cite{Bauhardt}, 
and apparently, the bound has not been improved since then. 
However, in the weaker form as in Theorem~\ref{thm:c-h}, this problem has been solved by Pietsch~\cite{Pietsch-Weyl} in 1980, see the final remarks there. In particular, it is shown that 
there is some $C_\alpha>0$ such that 
$\sup_{k\ge1} k^\alpha d_k(S) \le C_\alpha\, \sup_{k\ge1} k^{\alpha+1} h_k(S)$, see also~\cite[2.10.7]{Pie87} or~\cite[6.2.3.14]{Pie07}. 

\goodbreak

\textbf{The proof} 
of this result is the blueprint for the proof of 
Theorem~\ref{thm:c-h}. 
However, those proofs employ an intermediate comparison with the 
Weyl numbers $x_n$, which lie somehow 
between 
$h_n$ and $c_n$. 
Despite interesting consequences, 
this approach requires the multiplicativity of $x_n$, the notion of 2-summing norm, and some technical difficulties. 
The proof presented here is elementary: It 
only uses 
definitions and 
known properties of the determinant.

\textbf{Bounds for individual $n$} 
cannot be deduced from this approach, see also Remark~\ref{rem:eig}, 
and it remains a long-standing open problem if 
$d_{bn}(S)\le c\, n\, b_n(S)$, 
or even 
$a_{bn}(S)\le c\, n\, h_n(S)$, for some $b,c\ge1$, see~\cite[Prob.~5]{Pie09} or~\cite[2.10.7]{Pie87}.

\goodbreak 


\section{The proofs} \label{sec:proof}

\begin{proof}[Proof of Proposition~\ref{prop1}]
We refer to~\cite[2.11.9]{Pie87} for the proof that 
$s_n(S)=a_n(S)$ for any s-number sequence $(s_n)$, and any $S\in\calL(H,K)$ for Hilbert spaces~$H$ and~$K$. 
Just note that, for compact $S$, this follows quite directly from the singular value decomposition.
From this and (S3), we obtain 
\[
h_n(S) 
\,=\, \sup\bset{\frac{s_n(BSA)}{\|B\|\|A\|}\co A\in\calL(\ell_2,X), \; B\in\calL(Y,\ell_2), \; A,B\neq 0} 
\,\le\, s_n(S)
\] 
for any $(s_n)$. 
In addition, 
by (S2) and (S5), 
we obtain for any $L$ with $\rank(L)<n$ that
$
s_n(S)
\,\le\, s_n(L) + \|S-L\| 
\,=\, \|S-L\| 
$.
By taking the infimum over all such $L$, we see that 
$s_n(S)\le a_n(S)$. 
\end{proof}

\medskip

\begin{proof}[Proof of Theorem~\ref{thm:c-h}]
We first present the proof from~\cite[2.10.3]{Pie87} of the following statement: 
For fixed $\eps>0$, we can find 
$x_1,\dots,x_n\in B_X$ and $b_1,\dots,b_n\in B_{Y'}$ 
such that $\ipr{Sx_k,b_j}=0$ for $j<k$ and 
$(1+\eps)|\ipr{Sx_k,b_k}|>c_k(S)$ for $k=1,\ldots,n$. 

For this, we inductively assume that $x_k,b_k$ for $k<n$ are already found, and define 
\[
M_n \,:=\, \set{x\in X\co \ipr{Sx,b_k}=0 \text{ for } k<n}. 
\] 
Since $\codim{M_n}<n$, we can choose $x_n\in M_n \cap B_X$ 
with 
\[
(1+\eps) \|Sx_n\| \,\ge\, \|S J_{M_n}^X\| \,\ge\, c_n(S).  
\]
By the Hahn-Banach theorem, we choose $b_n\in B_{Y'}$ with $\ipr{Sx_n,b_n}=\|Sx_n\|\ge\frac{c_k(S)}{1+\eps}$. 

We now define the operators 
\[
A(\xi) \,:=\, \sum_{i=1}^n \xi_i x_i\in X, \quad \xi=(\xi_i)\in\ell_2^n,
\]
and
\[
B(y) \,:=\, \bigl(\ipr{y,b_i}\bigr)_{i=1}^n \in\ell_2^n, \qquad y\in Y, 
\medskip
\]
which satisfy $\|A\|,\|B\|\le \sqrt{n}$, 
and observe that $S_n:=BSA\co\ell_2^n\to\ell_2^n$ is generated by the triangular matrix $(\ipr{Sx_i,b_j})_{i,j=1}^n$ 
with determinant $\det(S_n)\ge 
\prod_{k=1}^n \frac{c_k(S)}{1+\eps}$. 

To obtain a bound with s-numbers, note that 
they all coincide for~$S_n$, esp.~with $a_k(S_n)$, and are equal to the 
singular numbers $s_k(S_n)$, i.e., the roots of the eigenvalues of $S_n S_n^*$, 
see~\cite[6.2.1.2]{Pie07}. 
As the determinant is multiplicative and equals the product of the eigenvalues, we see that 
$\det(S_n)=\sqrt{\det(S_nS_n^*)}=\prod_{k=1}^n a_k(S_n)$. 
%
From the definition of $h_n$, we obtain 
$a_k(S_n)\le \|A\|\|B\|\,h_k(S)\le n\cdot h_k(S)$, 
and hence 
\[\begin{split}
(1+\eps)^{-n}\prod_{k=1}^n c_k(S)  
\;&\le\;\det(S_n) 
\;=\;  \prod_{k=1}^n a_k(S_n) 
\;\le\;  n^n \prod_{k=1}^n h_k(S). 
\end{split}\]
With $\eps\to0$ and $c_n(S)\le \bigl(\prod_{k=1}^n c_k(S)\bigr)^{1/n}$ we obtain the result for $c_n(S)$. 

\goodbreak

The proof for $d_n(S)$ could be done via duality, 
at least for compact $S$, 
see e.g.~\cite[6.2.3.9 \& 6.2.3.12]{Pie07}. 
However, one can also prove it directly 
by inductively choosing $M_n:=\Span\{Sx_k\co k<n\}$, 
$x_n\in B_X$ with 
$(1+\eps) \|Q^Y_{M_n}Sx_n\| \,\ge\, \|Q^Y_{M_n}S\| \,\ge\, d_n(S)$, 
and $b_n\in B_{Y'}$ with 
$\ipr{Sx_n,b_n}=\|Q^Y_{M_n}Sx_n\|$ 
and  
$\ipr{Sx_k,b_n}=0$ for $k<n$. 
The remaining proof is as above. 
\end{proof}

\goodbreak

\begin{remark} \label{rem:eig}
The proof of Theorem~\ref{thm:c-h} uses the determinant 
to relate the eigenvalues $\lambda_k(S_n)$ of $S_n$ (which are the diagonal entries) with its singular numbers. 
Sometimes, the more general \textit{Weyl's inequality}~\cite{Weyl1949} from 1949 is used, which states that 
$
\prod_{k=1}^n |\lambda_k(S)| 
\;\le\; \prod_{k=1}^n a_k(S)
$
for any 
compact $S\in\calL(H,H)$, 
see also~\cite[3.5.1]{Pie07}.
This crucial step appears in all the proofs I am aware of 
that lead to the optimal factor~$n$ in the comparisons. 
Unfortunately, all these approaches use a whole collection of s-numbers, which does not allow for bounds for individual $n$. 

Let us present an example from~\cite[2.d.5]{Koe86} that shows that 
such product bounds 
between eigenvalues and s-numbers 
are to some extent best possible:\\
For $0<\sigma<1$, consider the matrix 
$T_n=(\delta_{j,i+1}+\sigma\cdot \delta_{i,n}\delta_{j,1})_{i,j=1}^n$ 
which represents a mapping on $\ell_2^n$. 
It is easy to verify that $a_k(T_n)=1$ for $k<n$ and $a_n(T_n)=\sigma$. 
(Recall that $a_k$ are the singular numbers in this case.)
Moreover, since $T_n^n=\sigma\cdot I_{\ell_2^n}$, 
we see that $|\lambda_k(T_n)|=\sigma^{1/n}$ for $k=1,\dots,n$. 
This shows that Weyl's inequality, 
as well as the easy corollary 
$|\lambda_n(S)|\le \|S\|^{1-\frac1n} a_n(S)^{1/n}$, 
are in general best possible. 
\end{remark}

\smallskip

\noindent 
{\bf Acknowledgement: } 
I thank Albrecht Pietsch for comments on an earlier version, and for encouraging me to make this note 
as self-contained as it is now. 
The last comments I received from him were on March 8, 2024, where he added that he was busy with his own article. 
My reply, however, was answered by his granddaughter, who informed me that Professor Pietsch had passed away on March 10. \\
I also thank S. Heinrich, A.~Hinrichs, D.~Krieg, T. K\"uhn, E.~Novak and the anonymous referees for helpful comments. 



\end{document}